\documentclass{amsart}
\usepackage{amssymb}
\usepackage[all]{xy}

\newtheorem{thm}{Theorem}[section]
\newtheorem{prop}[thm]{Proposition}

\theoremstyle{definition}
\newtheorem{defn}[thm]{Definition}
\theoremstyle{remark}
\newtheorem{rmk}[thm]{Remark}
\newtheorem{example}[thm]{Example}

\newcommand{\cB}{\mathcal{B}}
\newcommand{\cC}{\mathcal{C}}
\newcommand{\cD}{\mathcal{D}}
\newcommand{\cE}{\mathcal{E}}

\newcommand{\cH}{\mathcal{H}}
\newcommand{\cS}{\mathcal{S}}
\newcommand{\Fq}{\mathbb{F}_q}
\newcommand{\HH}{\mathbb{H}}
\newcommand{\N}{\mathbb{N}}
\newcommand{\Q}{\mathbb{Q}}
\newcommand{\Z}{\mathbb{Z}}
\newcommand{\Qlb}{\bar{\mathbb{Q}}_\ell}

\DeclareMathOperator{\Irr}{Irr}
\DeclareMathOperator{\ddet}{det}
\DeclareMathOperator{\supp}{supp}
\DeclareMathOperator{\codim}{codim}
\DeclareMathOperator{\pref}{{pf}}
\DeclareMathOperator{\rA}{A}
\DeclareMathOperator{\rB}{B}

\DeclareMathOperator{\rG}{G}
\DeclareMathOperator{\rI}{I}
\DeclareMathOperator{\GL}{GL}
\DeclareMathOperator{\GG}{G}
\DeclareMathOperator{\GO}{GO}
\DeclareMathOperator{\SO}{SO}

\title{Springer correspondences for dihedral groups}
\author{Pramod N. Achar}
\address{Department of Mathematics\\
  Louisiana State University\\
  Baton Rouge, LA 70803, USA}
\email{pramod@math.lsu.edu}
\author{Anne-Marie Aubert}
\address{Institut de Math\'ematiques de Jussieu\\
UMR 7586 du C.N.R.S.\\
F-75252 Paris Cedex 05
  \\
  France}
\email{aubert@math.jussieu.fr}
\date{\today}

\begin{document}

\begin{abstract}
Recent work by a number of people has shown that complex reflection groups
give rise to many representation-theoretic structures ({\it e.g.}, generic
degrees and families of characters), as though they were Weyl groups of
algebraic groups.  Conjecturally, these structures are actually describing
the representation theory of as-yet undescribed objects called
\emph{spetses}, of which reductive algebraic groups ought to be a special
case. 

In this paper, we carry out the Lusztig--Shoji algorithm for calculating
Green functions for the dihedral groups.  With a suitable set-up, the
output of this algorithm turns out to satisfy all the integrality and
positivity conditions that hold in the Weyl group case, so we may think of
it as describing the geometry of the ``unipotent variety'' associated to a
spets.  From this, we determine the possible ``Springer correspondences'',
and we show that, as is true for algebraic groups, each special piece is
rationally smooth, as is the full unipotent variety. 
\end{abstract}

\maketitle

\section{Introduction}
\label{sect:intro}

Many constructions arising in the representation theory of reductive
algebraic groups really depend only on the Weyl group.  In recent years, it
has been discovered that many of these constructions can be generalized to
the setting of complex reflection groups ({\it e.g.}, cyclotomic Hecke
algebras \cite{AK,A,BMR}, generic degrees~\cite{M}, root lattices~(on 
$\Z$~\cite{P}, on a ring of algebraic integers~\cite{N}),
families of characters~\cite{R,BK,K,MR}).  Indeed, it has been conjectured
that these constructions actually describe the representation theory of
some as-yet undescribed algebraic object called a \emph{spets}.  In this
paper, we add to this list by studying the ``geometry of the unipotent
variety'' associated to the dihedral groups, via the Lusztig--Shoji
algorithm for computing Green functions (see~\cite{S0}, \cite[\S 24]{L5}). 

It has already been observed by various people that this algorithm is
something that lends itself to generalization to complex reflection groups
(see~\cite{GM, S1, S2}).  We will review the algorithm in detail later, but
for now, let us simply recall that in order to carry out the algorithm for
a Weyl group $W$ of an algebraic group, one must first have some information
about the Springer correspondence for that group.  At a rudimentary level,
the required information is the partitioning of $\Irr(W)$ into disjoint
subsets, one for each unipotent class.

If we are working with a complex reflection group $W$, how do we choose such a
partition?  In~\cite{GM}, Geck and Malle used families of characters as the
subsets of the partition.  Of course, when $W$
is a Weyl group, this is quite different from the ``correct'' partitioning
by the Springer correspondence.  Nevertheless, it seems likely to have
geometric meaning: they conjectured that the algorithm in this form can be
used to compute the number of $\mathbb{F}_q$-points in a special piece of
the unipotent variety (see~\cite[Conjectures~2.5--2.7]{GM}, as well
as~\cite{L,SGM}). Separately, Shoji~\cite{S1,S2} has studied the algorithm
for imprimitive complex reflection groups by partitioning the characters
using combinatorial objects called ``symbols'' (these are generalizations
of the symbols and $u$-symbols that occur in the representation theory of
algebraic groups of classical type).

An important feature of the original Lusztig--Shoji algorithm is that its
output obeys certain integrality and positivity conditions. This is a
consequence of geometric considerations on the unipotent variety; the
algorithm itself, {\it a priori}, need not obey them.  However, Geck and
Malle conjecture that their version of the algorithm also satisfies these
properties, and Shoji proves that some of them hold for his version as
well. 

In this paper, we take a somewhat different approach: rather than fixing a
partition in advance, we consider all partitions subject to certain initial
constraints, and we seek to identify those (if any) for which the output of
the algorithm satisfies appropriate integrality and positivity conditions. 
For $W$ a finite complex reflection group and $\chi$ an irreducible
character of it, let $b_\chi$ denote the largest power of $q$ dividing the
fake degre of $\chi$. Then we
choose a cyclotomic Hecke algebra $\cH$ for $W$ and we fix a 
set $\cS$ of irreducible characters of $W$ including 
all \emph{special} characters. (Here $\chi$ is said to
be special if $b_\chi$ is equal to the largest power of $q$ dividing the generic 
degree of $\chi$, where the latter is defined with respect to 
the canonical symmetrizing trace on $\cH$.) We look for partitions of
$\Irr(W)$ into disjoint subsets $\cC$ such that (among other conditions) each 
$\cC$ contains a unique member (called the \emph{Springer character}
of $\cC$) on which $b$ attains its minimal value and the set of all
Springer characters is precisely $\cS$. In the case when $W$ is 
the Weyl group of a reductive connected algebraic group $G$ then we can
take for the Springer characters of $W$ all the irreducible characters which 
correspond, via the (actual) Springer correspondence, with the trivial local 
system on a unipotent class. In particular, this set of characters of $W$ is 
in bijection with the set of unipotent classes in $G$ and it depends
on $G$ itself.

Remarkably, when $W$ is a dihedral group, it turns out that in most cases (in 
all cases, under a minor additional condition), for a given set $\cS$, there is a 
\emph{unique} partition of the
desired sort.  In view of this uniqueness, it seems justified to refer to
that partition as ``the'' Springer correspondence with respect to  
$\cS$.  (Moreover, our
results are compatible with the various ``true'' Springer correspondences when the
dihedral group is a Weyl group of type $\rA_2$, $\rB_2$, or $\rG_2$---see
section~\ref{sec: compatibility}.) For that partition, we can then interpret 
the output
of the Lusztig--Shoji algorithm as giving information about the geometry of
a hypothetical ``unipotent variety'' for the dihedral group.  In
particular, we find that, as is true for algebraic groups, the unipotent
variety is rationally smooth~\cite{BM}, as is each special
piece~\cite{KP,L}.

We make all this precise in Section~\ref{sect:LS}, where we review the
Lusztig--Shoji algorithm, and define the various conditions and constraints
mentioned above.  In Section~\ref{sect:dihedral}, we review some basic
facts about the dihedral group, and we give precise statements (and
proofs for some) of the main results.  This section also includes a
tabulation of the output of the Lusztig--Shoji algorithm.  The one
remaining task is the proof of the uniqueness result mentioned above; this
occupies the entirety of Section~\ref{sect:proof}. 

For Weyl groups, the Springer characters are precisely those
which arise as $j_{W'}^W \chi$, where $j$ denotes truncated
induction, $W'$ is the stabilizer in $W$ of a point of the maximal torus,
and $\chi$ is a special character of $W'$. 
In~\cite{AA}, the authors have studied an analogue of this
procedure in the case of spetsial complex reflection groups, including the
dihedral groups.  
In this way, one obtains a preferred set $\cS_{\pref}$ of Springer
characters. The description of $\cS_{\pref}$ for the
dihedral groups will be recalled in Section~\ref{section: preferred Springer}.

On the other hand, it is very natural to consider the dihedral groups as a
class of non-crystallographic Coxeter groups. Then Kriloff
and Ram have proved in \cite[\S 3.4]{KR} that, as for Weyl groups, there exists a 
one-to-one correspondence between the irreducible characters of a dihedral
group $W$ and the 
tempered simple $\HH$-modules with real central characters, where $\HH$ is
a graded Hecke
algebra associated to $W$. Such a
correspondence, combined with the partition of irreducible characters of $W$ 
defined here, will provides a partition of the set of tempered simple 
$\HH$-modules with real central characters. 

\section{The Lusztig--Shoji algorithm}
\label{sect:LS}

Let $W$ be a finite complex reflection group acting on a vector space $V$.  Let
$q$ be an indeterminate and 
$P_W(q)$ denote the Poincar\'e polynomial of $W$: 
\[
P_W(q) = \prod_{i=1}^{\dim V} \frac{q^{d_i} - 1}{q-1},
\]
where $d_1, \ldots, d_{\dim V}$ are the degrees of $W$.  Next, for any
class function $f$ of $W$, we define a polynomial $R(f)$ by
\[
R(f)(q) = \frac{(q-1)^{\dim V}P_W(q)}{|W|} \sum_{w \in W} \frac{\det_V(w)f(w)}{\det_V(q\cdot\mathrm{id}_V - w)}.
\]
If $\chi$ is an irreducible character of $W$, then $R(\chi)$ is commonly
known as the \emph{fake degree} of $\chi$.  In this case, we define
$b_\chi$ to be the largest power of $q$ dividing $R(\chi)$. 
We note that (see \cite[(6.2)]{M2})
\begin{equation} \label{eqn: symetry}
R(\overline{{\ddet_{V}} \otimes f})(q)=q^{N^*}R(f)(q^{-1}).
\end{equation}
(Here $N^*$ is the number of reflections in $W$, and
$^{-}$ is
the complex conjugation.  Of course, for a
Coxeter group, $\det_V$ and $\overline\det_V$ are both just the sign
character. Hence the equation~(\ref{eqn: symetry}) is a generalization 
to complex reflection groups of \cite[Proposition~11.1.2]{C}.)

Next, let $\Omega$ be the square matrix with rows and columns indexed by
$\Irr(W)$ defined as follows: 
\[
\Omega = (\omega_{\chi,\chi'})_{\chi,\chi' \in \Irr(W)},
\qquad
\omega_{\chi,\chi'} = q^{N^*}R(\chi \otimes \chi' \otimes
\overline\det_V).
\]

\begin{defn}
A \emph{Lusztig--Shoji datum} for a complex reflection group $W$ is a
triple $(X, <, a)$, where $X = \{\cC\}$ is a partition of $\Irr(W)$ into
disjoint subsets (that is, $\Irr(W) = \bigsqcup_{\cC \in X} \cC$); the
relation $<$ is a total order on $X$; and $a: X \to \N$ is an order
reversing function.  The member of $X$ to which a given $\chi \in \Irr(W)$ belongs is
called its \emph{support}, and the statement $\supp \chi = \cC$ is
equivalent to the statement that $\chi \in \cC$. 
\end{defn}

\begin{defn} \label{defn: Green}
A \emph{system of Green functions} with respect to a Lusztig--Shoji datum
$(X, <, a)$ is a solution to the matrix equation 
\begin{equation}\label{eqn:LS}
P \Lambda P^t = \Omega,
\end{equation}
where $P$ and $\Lambda$ are also square matrices of rational functions over 
$\Z[q]$ with
rows and columns indexed by $\Irr(W)$, subject to the following
conditions: 
\begin{equation}\label{eqn:PLdefn}
\begin{aligned}
P_{\chi,\chi'} &= 
\begin{cases}
0 & \text{if $\supp \chi < \supp \chi'$,} \\
\delta_{\chi,\chi'}q^{a_\cC} & \text{if $\supp \chi = \supp \chi' = \cC$,}
\end{cases}
\\
\Lambda_{\chi,\chi'} &= 0 \qquad \text{if $\supp \chi \ne \supp \chi'$.}
\end{aligned}
\end{equation}
\end{defn}

The following notation for picking out certain submatrices of these
matrices will be useful:
\begin{align*}
P_{\chi,\cC} &= (P_{\chi,\chi'})_{\chi' \in \cC} &
\Lambda_\cC &= (\Lambda_{\chi,\chi'})_{\chi,\chi' \in \cC} &
\Omega_{\cC,\cC'} &= (\omega_{\chi,\chi'})_{\chi \in \cC, \chi' \in \cC'}
\end{align*}

We have the following fundamental fact:

\begin{prop}[Lusztig, Shoji, Geck--Malle]
Every Lusztig--Shoji datum admits a unique system of Green functions.
\end{prop}

For a proof, see \cite[Proposition~2.2]{GM}.  (In {\it loc.~cit.}, the
proposition is stated only for finite Coxeter groups, and only for a
certain specific Lusztig--Shoji datum, but the proof is in fact completely
general.)  In the course of the proof, one obtains an inductive formula for
computing $P$ and $\Lambda$, as follows: Given $\cC \in X$, suppose that
the blocks $P_{\chi,\cC'}$ and $\Lambda_{\cC'}$ are known for all $\cC' <
\cC$ and all $\chi \in \Irr(W)$.  Then $\Lambda_\cC$ and $P_{\chi,\cC}$ are
given by
\begin{align}
\Lambda_\cC &= q^{-2a_\cC}\left(\Omega_{\cC,\cC} - \sum_{\cC' < \cC}
P_{\cC,\cC'} \Lambda_{\cC'} P_{\cC,\cC'}^t \right)
\label{eqn:Lambda-alg}\\ 
P_{\chi,\cC} &= q^{-a_\cC}\left(\Omega_{\chi,\cC} - \sum_{\cC' < \cC}
P_{\chi,\cC'} \Lambda_{\cC'} P_{\cC,\cC'}^t \right) \Lambda_\cC^{-1}
\qquad\text{if $\supp \chi > \cC$.} \label{eqn:P-alg}
\end{align}
(Of course, if $\supp \chi \le \cC$, then $P_{\chi,\cC}$ is determined
by~\eqref{eqn:PLdefn}.)

We choose a cyclotomic Hecke algebra $\cH$ for $W$ (see~\cite{BMR}). Then
$\cH$ admits a canonical symmetrizing trace. Indeed, the form defined in 
\cite{BrMa} satisfies (1)(a) and (1)(b) of \cite[Theorem-Assumption~2]{BMM} 
(the fact that it turns $\cH$ into a symmetric algebra is proved in
\cite{MM}); on the other hand, since it follows from \cite{GIM} that the 
corresponding Schur elements are as conjectured in \cite[Vermutung~1.18]{M}, 
the form satisfies also (1)(c) of \cite[Theorem-Assumption~2]{BMM} by
\cite[Lemma~2.7]{BMM}.  Such a form is unique. 

Then one can consider the
generic degrees of the characters of $W$ defined with respect to $\cH$ and the
above symmetrizing trace.  
Recall that a character $\chi \in \Irr(W)$ is said to be \emph{special} if
$b_\chi$ is equal to the largest power of $q$ dividing the generic degree
of $\chi$.  
In addition, $\cH$ determines the
families of characters of $\Irr(W)$ (see~\cite{R}, \cite{MR}), which play a role
below.

In the following definition, we list a number of desirable properties that
a Lusztig--Shoji datum may have.  

\begin{defn}\label{defn:springer}
Choose a cyclotomic Hecke algebra $\cH$ for $W$.  A Lusztig--Shoji datum
$(X, <, a)$ is said to be a \emph{Springer correspondence} for $W$ and
$\cH$ if the following additional conditions are satisfied: 
\begin{enumerate}
\item For each $\cC$, we have $a_\cC = \min \{ b_\chi \mid \chi \in \cC
\}$.  Moreover, there is a unique member of $\cC$ (called the
\emph{Springer representation} of $\cC$) on which $b$ attains its minimum
value. 
\item Every special representation of $W$ occurs as a Springer
representation of some $\cC \in X$. 
\item If $\chi'$ is a nonspecial representation in the same family as the
special representation $\chi$, then $\supp \chi' \le \supp \chi$.
\item The entries of $\Lambda$ are polynomials with integer coefficients,
and the entries of $P$ are polynomials with nonnegative integer
coefficients. 
\item If $\chi \in \cC$, then $P_{\chi,\chi'}$ is divisible by $q^{a_\cC}$
for all $\chi'$.
\end{enumerate}
In this case, the sets $\cC$ are called \emph{unipotent classes}.  A class
is \emph{special} if its Springer representation is. 
\end{defn}

\begin{rmk}
Let $W$ be the Weyl group of a reductive connected algebraic group $G$.
We assume that $G$ is defined over $\Fq$ with Frobenius map $F$ and that
there exists a maximal torus $T$ of $G$ which is defined and split over
$\Fq$. 
Recall that the (actual) Springer correspondence is a map $\nu$ from $\Irr(W)$ 
to the set of pairs $(C,\cE)$, where $C$ is a unipotent $G$-conjugacy class
in $G$ and $\cE$ is a $G$-equivariant irreducible local system on $C$.
Let $X$ be the partition of $\Irr(W)$ induced by $\nu$, that is,
$X=\{\cC\}$, where $\cC=\cC_C$ is the set of $\chi\in\Irr(W)$ such that
$\nu(\chi)=(C,\cE)$ for some $G$-equivariant irreducible local system $\cE$ 
on $C$. Let $a\colon X\to\N$ be the function defined by 
$a_{\cC}:=\dim\cB_u$, where $\cB_u$ is the variety of Borel subgroups of
$G$ containing $u\in C$.

All the above properties (1)--(5) hold in the this case:
\begin{itemize}
\item[(1)]
If $\chi\in\cC$ then we have $a_\cC=\dim\cB_u\le b_\chi$ (see \cite[\S
1.1]{Sp1}). Moreover, if $\nu(\chi)=(C,\Qlb)$ then we have $a_\cC=b_\chi$ 
(see \cite[Corollary~4]{BM0}). Hence (1) is satisfied 
and the ``Springer representations'' of $W$ are the
irreducible representations of $W$ which correspond, via the (actual) Springer
correspondence, with the trivial local system on a unipotent class. In
particular, the set of Springer representations is in bijection with the set of
unipotent classes. 
\item[(2)]
For $\chi\in\Irr(W)$ let $a'_\chi$ denote the largest power of $q$
dividing the generic degree of $\chi$. We have $a_\chi'\le a_\cC$ if
$\chi\in\cC$ (see \cite[Cor.~10.9]{LUS}). In particular, when $\chi$ is
special, it then follows from (1) that $a_\chi'=a_\cC=b_\chi$.
\item[(3)] 
Property~(3) is satisfied (see \cite[Proposition~2.2]{GM1}). 
\item[(4)] 
Let $\widetilde\Omega=(\widetilde\omega_{\chi,\chi'})_{\chi,\chi'\in\Irr(W)}$, 
where
\[\widetilde\omega_{\chi,\chi'}=q^{\dim T-{1/2}(\codim C+\codim C')}\cdot
\frac{|G^F|}{|W|}\cdot
\sum_{w\in W}\chi(w)\chi'(w)\cdot|T^F_w|^{-1},\]
where $\nu(\chi)=(C,\cE)$ and $\nu(\chi')=(C',\cE')$.

Using the fact that $a_\cC={1/2}(\codim C-\dim T)$ (see
\cite[(5.10.1)]{C}), we see that
\[\omega_{\chi,\chi'}=q^{a_\cC+a_{\cC'}}\cdot\widetilde\omega_{\chi,\chi'}.\]
Lusztig proved in \cite[\S 24]{L5} (see especially \cite[(24.5.2)]{L5})
that the equation $\widetilde P\Lambda\widetilde
P^t=\widetilde\Omega$ with respect to the unknown variables
$\Lambda$, $\widetilde P$ admits a solution with
$\Lambda_{\chi,\chi'}\in\Z[q]$, $\widetilde
P_{\chi,\chi'}\in\Z[q]$ and
\[\begin{aligned}
\widetilde P_{\chi,\chi'} &= 
\begin{cases}
0 & \text{if $\supp \chi < \supp \chi'$,} \\
1 & \text{if $\supp \chi = \supp \chi' = \cC$,}
\end{cases}
\\
\Lambda_{\chi,\chi'} &= 0 \qquad \text{if $\supp \chi \ne \supp \chi'$.}
\end{aligned}
\]
Then the matrices $P$ and $\Lambda$ with
$P_{\chi,\chi'}=q^{a_\cC}\widetilde P_{\chi,\chi'}$ where
$\nu(\chi)=(C,\cE)$ satisfy the condition~(\ref{eqn:PLdefn})
in Definition~\ref{defn: Green}.
\item[(5)] Since $\widetilde
P_{\chi,\chi'}\in\Z[q]$, we have that $P_{\chi,\chi'}$ is divisible 
by $q^{a_\cC}$ for all $\chi'\in\Irr(W)$.
\end{itemize}
\end{rmk} 

\smallskip

Any system of Green functions gives rise to a partial order $\preceq$ on
$X$ that is compatible with, but in general weaker than, the order $<$, as
follows: $\preceq$ is the transitive closure of the relation 
\[
\cC \preceq \cC'
\qquad\text{if there exist $\chi \in \cC$ and $\chi' \in \cC'$ such that
$P_{\chi,\chi'} \ne 0$}.
\]
In the case of a Springer correspondence, we call this the \emph{closure
order} on unipotent classes.  A \emph{special piece} is then defined just
as for algebraic groups: it is the union of a special class and all those
nonspecial classes in its closure that are not also in the closure of any
smaller special class. 

\section{The dihedral groups}
\label{sect:dihedral}

Our work on the dihedral groups will take place in the following framework:
we fix a set $\cS \subset \Irr(W)$, including all special characters, and
we look for Springer correspondences $(X,<,a)$ whose set of Springer
representations is precisely $\cS$. 

\begin{rmk}
It seems reasonable to fix the set of Springer representations in advance
because for Weyl groups, there is an elementary way to compute the set of
Springer representations without knowing the full Springer correspondence. 
(The set of Springer representations is precisely the set of
representations arising as $j_{W'}^W \chi$, where $j$ denotes truncated
induction, $W'$ is the stabilizer in $W$ of a point of the maximal torus,
and $\chi$ is a special character of $W'$, see \cite[\S~12.6]{C}.)

An analogue of this procedure in the dihedral groups provides 
a preferred set $\cS_{\pref}$ of Springer
representations, to which one can apply the results obtained here
(see section~\ref{section: preferred Springer}).
\end{rmk}

Henceforth, we work only with the dihedral group $W = \rI_2(m)$.  We begin by
recalling some facts about the representation theory of $\rI_2(m)$.  Its
irreducible representations are: 
\[
\chi_0, \chi_1, \ldots, \chi_{\lfloor\frac{m-1}{2}\rfloor}, \epsilon;
\qquad\text{and}\qquad
\chi_r, \chi_r' \quad\text{if $m = 2r$.}
\]
Here, $\chi_0$ is the trivial representation, $\epsilon$ is the sign
representation, and we have
\[
\text{$b_{\chi_i} = i$ for all $i$},\qquad
\text{$b_{\chi_r'}=r$,}
\qquad\text{and}\qquad
b_\epsilon = m.
\]
The representations $\chi_1, \ldots, \chi_{\lfloor(m-1)/2\rfloor}$ are all
$2$-dimensional, while $\chi_0$, $\epsilon$, and $\chi_r$ and $\chi_r'$ are
$1$-dimensional.  The special representations are $\chi_0$, $\chi_1$, and
$\epsilon$. 

The matrix $\Omega$ is described in the following table.  Recall that
$\Omega$ is symmetric; below, to reduce clutter, we have only recorded the
part of $\Omega$ below the diagonal. In the table below $i$ and $j$ are
assumed to be non-zero.
{\small
\[
\begin{array}{c|ccccc}
\Omega & \chi_0 & \chi_i & \chi_r & \chi_r' & \epsilon \\ \hline
\chi_0   & q^{2m} &&&& \\
\chi_j & q^{m+j}+q^{2m-j} &  
q^{m+|i-j|} + q^{m+i+j}  + q^{2m-i-j} + q^{2m-|i-j|} &&& \\
\chi_r  & q^{\frac{3}{2}m} & q^{\frac{3}{2}m-i} + q^{\frac{3}{2}m+i} &
q^{2m} && \\
\chi_r' & q^{\frac{3}{2}m} & q^{\frac{3}{2}m-i} + q^{\frac{3}{2}m+i} &
q^{m} & q^{2m} & \\
\epsilon& q^m & q^{m+i}+q^{2m-i} & q^{\frac{3}{2}m} & q^{\frac{3}{2}m} &
q^{2m}
\end{array}
\]}

Let $\cS$ be a set of irreducible representations of $\rI_2(m)$ including $\chi_0$, $\chi_1$, and $\epsilon$.  In case $m$ is even and $\cS$ contains exactly one of $\chi_r$ and $\chi_r'$, we assume henceforth, without loss of generality, that it in fact contains $\chi_r'$.  (This assumption will allow us to simply some formulas by treating $\chi_r$ and the various $\chi_i$ with $i < r$ uniformly.)

Let us define a sequence of integers
\begin{equation} \label{eq: sequence}
d_0 < d_1 < \cdots < d_N < m/2
\qquad\text{where}\qquad
\text{$d_0 = 0$ and $d_1 = 1$}
\end{equation}
by $\{\chi_i \in \cS \mid i < m/2\} = \{\chi_{d_0}, \ldots, \chi_{d_N}\}$.  We will show that $\rI_2(m)$ admits a Springer correspondence whose set of Springer representations is precisely $\cS$.  We will use the following notation for unipotent classes:
\[
\begin{array}{ll}
\cC_k: & \text{class with Springer representation $\chi_{d_k}$} \\
\cC_{\chi_r}, \cC_{\chi_r'}, \cC_\epsilon: & \text{classes with Springer representations $\chi_r$, $\chi_r'$, $\epsilon$, respectively}
\end{array}
\]
Note that $\cC_\epsilon$ is automatically a singleton, since there are no characters $\chi$ with $b_\chi > b_\epsilon$.  Similarly, if $\chi_r' \in \cS$, then $\cC_{\chi_r'}$ must be a singleton: the only character $\chi$ with $b_\chi > b_{\chi_r'}$ is $\epsilon$, which is already the Springer representation of another class.  The same argument applies to $\cC_{\chi_r}$ if $\chi_r \in \cS$.

\begin{thm}\label{thm:main}
Let $\cS$ and $d_0, \ldots, d_N$ be as above.  $\rI_2(m)$ admits a Springer correspondence whose set of Springer representations is precisely $\cS$.  Every such Springer correspondence has the form
\begin{align*}
\cC_0 &= \{\chi_0\} \\
\cC_k &= \{\chi_{d_k}, \chi_{d_k+1}, \ldots, \chi_{d_{k+1}-1} \} \cup
\{ \chi_{f_{k+1}+1}, \chi_{f_{k+1}+2}, \ldots, \chi_{f_k} \} \;\;
\text{for $1\le k\le N-1$}\\
\cC_N &=\begin{cases}
\{\chi_{d_N}, \chi_{d_N+1}, \ldots, \chi_{\lfloor(m-1)/2\rfloor}
\}&\text{if $m$ odd, or $m$ even and $\chi_r, \chi_r' \in \cS$}\\
\{\chi_{d_N}, \chi_{d_N+1}, \ldots, \chi_{f_N} \}&
\text{if $m$ even, $\chi_r\notin \cS$ and $\chi'_r\in \cS$}\\
\{\chi_{d_N}, \chi_{d_N+1}, \ldots, \chi_{r-1}, \chi_r,\chi'_r \}&
\text{if $m$ even and $\chi_r, \chi_r' \notin \cS$}
\end{cases}\\
\cC_{\chi_r} &= \{\chi_r\} \qquad\text{if $\chi_r \in \cS$} \\
\cC_{\chi'_r} &= \{\chi'_r\} \qquad\text{if $\chi'_r \in \cS$} \\
\cC_{\epsilon} &= \{\epsilon\}
\end{align*}
for a suitable sequence of integers $f_1 \ge f_2 \ge \cdots \ge f_N$.  (It is possible that $f_k = f_{k+1}$, in which case the second part of $\cC_k$ is empty.)

Moreover, except in the case where $m$ is even and $\cS$ contains exactly one of $\chi_r$ and $\chi_r'$, we actually have $f_k = \lfloor \frac{m-1}{2} \rfloor$ for all $k$, so $\rI_2(m)$ admits a unique Springer correspondence whose set of Springer representations is $\cS$.

On the other hand, if $m$ is even and $\cS$ contains $\chi_r'$ but not $\chi_r$, then for any integers $r = f_1 \ge \cdots \ge f_N \ge d_N$ such that $f_k - f_{k+1} \le d_{k+1} - d_k$ for each $k$, there is a Springer correspondence as above.
\end{thm}

We introduce the following notation: if $\supp \chi \ge \cC$ and $\supp \chi' \ge \cC$, then let
\begin{equation}\label{eqn:Ydefn}
Y^\cC_{\chi,\chi'} = \gamma^{-1}\left(\omega_{\chi,\chi'} - \sum_{\cC' < \cC}
P_{\chi, \cC'} \Lambda_{\cC'} P_{\chi',\cC'}^t\right)
\qquad\text{where}\qquad
\gamma = q^m - 1
\end{equation}
It is then clear from~\eqref{eqn:Lambda-alg} that
\begin{equation}\label{eqn:Lambda-Y}
\Lambda_{\cC} = q^{-2a_\cC} \gamma \left(Y^\cC_{\chi,\chi'}\right)_{\chi,\chi' \in \cC}.
\end{equation}
Below, we will give formulas for $Y^{\cC}$ and $P$.  It is then immediate to compute $\Lambda$.

Henceforth, for the sake of brevity of notation, we will generally write ``$i$'' instead of ``$\chi_i$,'' as well as ``$r'$'' for ``$\chi_r'$.''  We also write $Y^{(k)}$ for $Y^{\cC_k}$. 

Let
\[
\iota =
\begin{cases}
1 & \text{if $m$ is odd, or if $m$ is even and $r, r' \notin \cS$,} \\
0 & \text{if $m$ is even, $r' \in \cS$, and $r \notin \cS$,} \\
-1& \text{if $m$ is even and $r, r' \in \cS$.}
\end{cases}
\]
In Section~\ref{sect:proof}, we will establish the following formulas for $Y^{(k)}$ by induction on $k$:
\begin{gather*}
Y^{(N)}_{ij} = q^{m-|i-j|} + \iota q^{i+j} 
\qquad\qquad
Y^{(N)}_{ir'} =
\begin{cases}
q^{i+r} & \text{if $i < r$,} \\
0       & \text{if $i = r$}
\end{cases} \\
Y^{(k)}_{ij} =
\begin{cases}
q^{m-|i-j|} - q^{m+i+j-2d_{k+1}} & \text{if $i,j < d_{k+1}$,} \\
0 & \text{if $i< d_{k+1} \le f_{k+1} < j$} \\
  & \quad\text{or $j < d_{k+1} \le f_{k+1} < i$,} \\
q^{m-|i-j|} - q^{m-i-j+2f_{k+1}} & \text{if $i,j > f_{k+1}$.}
\end{cases}
\end{gather*}
We will simultaneously show that $P$ is given by
\[
P_{ij} =
\begin{cases}
q^i & \text{if $i < j$ and $j \in \cS$} \\
q^{d_k+f_k-i} & \text{if $i > j = f_k$ for some $k$} \\
0 & \text{otherwise}
\end{cases}
\qquad
\begin{aligned}
P_{ir'} &=
\begin{cases}
q^i & \text{if $r' \in \cS$ and $i < r$} \\
0   & \text{otherwise}
\end{cases} \\
P_{\chi,\epsilon} &= R(\chi)
\end{aligned}
\]
(Here, the formula for $P_{ij}$ is only valid under the assumption that
$\supp i > \supp j$.)

\begin{table} \label{table: closure order}
\begin{center}
\begin{tabular}{c@{\hspace{2cm}}c@{\hspace{2cm}}c}
$\vcenter{\xymatrix@=5pt{
\cC_0 \ar@{-}[d] \\
\vdots \ar@{-}[d] \\
\cC_N \ar@{-}[d] \\
\cC_\epsilon}}$
&
$\vcenter{\xymatrix@R=5pt@C=2.5pt{
&\cC_0 \ar@{-}[d] \\
&\vdots \ar@{-}[d] \\
&\cC_N \ar@{-}[dl] \ar@{-}[dr] \\
\cC_{\chi_r} \ar@{-}[dr] && \cC_{\chi_r'} \ar@{-}[dl] \\
&\cC_\epsilon}}$
&
$\vcenter{\xymatrix@=5pt{
\cC_0 \ar@{-}[d] \\
\vdots \ar@{-}[d] \\
\cC_N \ar@{-}[d] \\
\cC_{\chi_r'} \ar@{-}[d] \\
\cC_\epsilon}}$
\\
$m$ odd, or  & $m$ even, & $m$ even, \\
$m$ even and $r, r' \notin \cS$ & $r, r' \in \cS$ &
$r' \in \cS$, and $r \notin \cS$
\end{tabular}
\end{center}
\medskip
\caption{Closure order of unipotent classes in dihedral groups}
\end{table}

Now, recall that a variety $Z$ is \emph{rationally smooth}
(see~\cite[Appendix]{KL}) if
\[
\mathcal{H}^i_z\mathrm{IC}(Z,\bar\Q_l) =
\begin{cases}
\bar\Q_l & \text{if $i = 0$,} \\
0        & \text{otherwise}
\end{cases}
\qquad\text{for all $z \in Z$.}
\]
In the original Lusztig--Shoji algorithm, the entries of $P$ enjoy the
following interpretation in terms of intersection cohomology complexes: if
$\chi$ corresponds to the local system $\cE$ on the unipotent class $C$,
and likewise $\chi'$ corresponds to $\cE'$ on $C'$, then we have
\begin{equation}\label{eqn:IC}
P_{\chi,\chi'} = \sum_{i \ge 0} [ \cE' : \mathcal{H}^i\mathrm{IC}(\bar C,
\cE) ] q^{a(C) + i},
\end{equation}
where $a(C)$ is the dimension of the Springer fiber over a point of $C$. 
(See~\cite[Theorem 24.8]{L5}, but note that the definition of $\Omega$ is
slightly different there, resulting in a slightly different formula for
$P$.)

If we now interpret the entries of $P$ for a dihedral group
via~\eqref{eqn:IC} as describing the geometry of some unknown variety, we
obtain the following result.

\begin{thm}\label{thm:smooth}
With respect to any Springer correspondence, every special piece of $\rI_2(m)$ is rationally smooth.  The full unipotent variety is also rationally smooth.
\end{thm}
\begin{proof}
The classes $\cC_0$ and $\cC_\epsilon$ each constitute a special piece by themselves, so they are obviously rationally smooth.  For the ``middle'' special piece, which contains the special class $\cC_1$, we see that $P_{1d_k} = q$ for all $k$, and $P_{1r'} = q$, but $P_{1i} = 0$ if $i$ is not a Springer representation, so this piece is rationally smooth as well.

The second sentence is simply the observation that $P_{0i} = 1$ if $i$ is a Springer representation, and $P_{0i} = 0$ otherwise.
\end{proof}

Now, suppose a group admits multiple Springer correspondences for a fixed
set $\cS$ of Springer representations.  Since $\cS$ is fixed, it is
possible to identify unipotent classes in distinct Springer
correspondences, and it makes sense to compare the support of a given $\chi
\in \Irr(W)$ in various Springer correspondences.

In this situation, a Springer correspondence $X$ is called \emph{maximal}
if $\supp_X \chi \ge \supp_{X'} \chi$ for all $\chi \in \Irr(W)$ and all
other Springer correspondences $X'$ for $\cS$.  There is some
evidence (see Remark~\ref{rmk:max} below) that the actual Springer
correspondences for algebraic groups satisfy a maximality condition of this
type; perhaps it is something that should be added to
Definition~\ref{defn:springer}.  In any case, for the dihedral groups, if
we require maximality, then $\cS$ determines a unique Springer
correspondence in all cases.

\begin{thm}\label{thm:max}
For any set $\cS$ of irreducible representations of $\rI_2(m)$ with $\chi_0,
\chi_1, \epsilon \in \cS$, there is a (necessarily unique) maximal Springer
correspondence for $\rI_2(m)$ whose set of Springer representations is $\cS$.
\end{thm}
\begin{proof}
There is only something to prove in the case that $m$ is even, $r' \in \cS$, and $r \notin \cS$.  We define $f_k$ inductively.  Let $f_1 = r$, and then for $k > 1$, let $f_k = \max\{d_N, f_{k-1} - (d_k - d_{k-1})\}$.
\end{proof}

\begin{rmk} \label{rmk:max}
Note that the uniqueness proved in~Theorem~\ref{thm:max} is not of "global
nature", in the sense that it depends on the choice of
the set $\cS$ (the set of Springer representations in our terminology). In
the case of the dihedral group of order $8$, that is, the Weyl group of type
$\rB_2$, this set will not be the same for the Springer correspondences associated
to groups in odd characteristic, characteristic $2$ or to disconnected
groups (see section~\ref{sec: compatibility}). Hence we get more than one 
Springer correspondence for the group $\rB_2$.
\end{rmk}
 
\section{Compatiblity with known Springer correspondences} \label{sec:
compatibility}
\subsection{Connected algebraic groups in odd characteristic}
In the cases $m = 3, 4, 6$, when $\rI_2(m)$ is in fact the Weyl group of a
connected algebraic group of type $\rA_2$, $\rB_2$, $\rG_2$ respectively, 
it is easy to verify that the unique maximal Springer
correspondence of Theorem~\ref{thm:max} coincides with the ``true''
Springer correspondence, as found in, say,~\cite[Section~13.3]{C}.  The
Springer correspondences for these three groups are given below.

For algebraic groups of types $\rA_2$ and $\rB_2$, both the unipotent classes and the
representations of the Weyl group are labelled by partitions (or pairs of
partitions).  In $G_2$, unipotent classes are named by their Bala--Carter
labels, and representations have been named using Carter'
s notation~\cite{C}.  To identify representations in these notations with
our $\chi_i$'s, it suffices to compute $b_\chi$ for all of them. In types
$\rA_2$ and $\rB_2$, this can be done with~\cite[Propositions~11.4.1
and~11.4.2]{C}, while for $\rG_2$, the required information is included in
the notation (we have $b_{\phi_{d,n}} = n$).  In types $\rB_2$ and $\rG_2$,
one has a choice of which representation to label as $\chi_r$ and which as
$\chi_r'$; we have made the choice that agrees with Theorem~\ref{thm:main}.
The corresponding sets $\cS$ are
$\{\chi_0,\chi_1,\epsilon\}$, $\{\chi_0,\chi_1,\chi'_2,\epsilon\}$,
$\{\chi_0,\chi_1,\chi_2,\chi'_3,\epsilon\}$, in type $\rA_2$, $\rB_2$,
$\rG_2$, respectively.
{\tiny
\[
\begin{array}[t]{c|c}
\multicolumn{2}{c}{\text{Type $\rA_2$}} \\
\text{\it class} & \text{\it reps.} \\ \hline
{}[3] & [3] = \chi_0 \\
{}[2,1] & [2,1] = \chi_1 \\
{}[1^3] & [1^3] = \epsilon
\end{array}
\qquad
\begin{array}[t]{c|c}
\multicolumn{2}{c}{\text{Type $\rB_2$}} \\
\text{\it class} & \text{\it reps.} \\ \hline
{}[5] & ([2],\varnothing) = \chi_0 \\
{}[3,1^2] & ([1],[1]) = \chi_1;\ ([1^2],\varnothing) = \chi_2 \\
{}[2^2,1] & (\varnothing,[2]) = \chi'_2 \\
{}[1^5] & (\varnothing,[1^2]) = \epsilon
\end{array}
\qquad
\begin{array}[t]{c|c}
\multicolumn{2}{c}{\text{Type $\rG_2$}} \\
\text{\it class} & \text{\it reps.} \\ \hline
G_2 & \phi_{1,0} = \chi_0 \\
G_2(a_1) & \phi_{2,1} = \chi_1;\ \phi_{1,3}' = \chi_3 \\
\tilde A_1 & \phi_{2,2} = \chi_2 \\
A_1 & \phi_{1,3}'' = \chi'_3 \\
1 & \phi_{1,6} = \epsilon
\end{array}
\]}

\subsection{Bad characteristics}
It follows from \cite{Sp0} that for connected algebraic groups of both types 
$\rB_2$ in characteristic $2$ and $\rG_2$ in characteristic $3$, the partially 
ordered set of unipotent classes has the form
of second diagram in Table~$1$ of the present paper (the one in which the two
classes
$\cC_{\chi_r}$ and $\cC_{\chi_r'}$ cannot be compared). Then, the explicit
computations of the Springer correspondences in these two cases, provided
by \cite{LS} and \cite{Sp1}, show that they coincide with 
those given by Theorem~\ref{thm:main}.
The same references show that the Springer correspondence for $\rG_2$ in
characteristic $2$ coincides with the one in good characteristic.
Hence the Springer correspondences associated to reductive groups of rank
$2$ in bad characteristic are recovered by our theorem.

\subsection{Non-connected algebraic groups}
We will consider below two Springer correspondences which occur ``in nature,''
associated to the group $\rB_2$ and to two types of disconnected
groups. 

We assume here that the characteristic equals $2$.
Let $\GG(5)$ be the group defined as the extension
of $\GL(5)$ by the order $2$ automorphism of the diagram. There are $5$
unipotent classes of $\GG(5)$ which are not contained in $\GL(5)$, and
the partial order is here of the kind of the second diagram of 
Table~$1$ of section~\ref{sect:dihedral} of the present paper.
Then Table $2$ on page $314$ of \cite{MS}
shows that the Springer correspondence for this group coincides also with
the
one given by Theorem~\ref{thm:main}. 
We have $\cS=\{\chi_0,\chi_1,\chi_2,\chi'_2,\epsilon\}$.

Let $\GO_6$ be the general orthogonal group over an algebraically closed
field of characteristic $2$, the extension of $\SO_6$ by a non-trivial
graph automorphism of order $2$.
There are $4$ unipotent classes which are not
contained in $\SO_6$, the corresponding partial order between them is of
the
form of the third diagram of Table~$1$ of section~\ref{sect:dihedral}
of the present paper, and Table $4$
on page $318$ of \cite{MS} shows that the restricted Springer correspondence
for this group coincides also with the one deduced from 
Theorem~\ref{thm:main}.
We have $\cS=\{\chi_0,\chi_1,\chi'_2,\epsilon\}$.

The Springer correspondences for the two above groups are given below,
where unipotent classes are labelled as in \cite{MS}.
{\tiny
\[
\begin{array}[t]{c|c}
\multicolumn{2}{c}{\text{ Group $\GG(5)$}} \\
\text{\it class} & \text{\it reps.} \\ \hline
{}[5] & ([2],\varnothing) = \chi_0 \\
{}[3,1^2] & ([1],[1]) = \chi_1 \\
{}[3,1^2_0] & ([1^2],\varnothing) =\chi_2\\
{}[2^2,1] & (\varnothing,[2])=\chi'_2\\
{} [1^5] & (\varnothing,[1^2])= \epsilon
\end{array}
\qquad\qquad
\begin{array}[t]{c|c}
\multicolumn{2}{c}{\text{ Group $\GO_6$}} \\
\text{\it class} & \text{\it reps.} \\ \hline
{}[6] & ([2],\varnothing) = \chi_0 \\
{}[4,1^2] & ([1],[1]) = \chi_1;\ ([1^2],\varnothing) = \chi_2 \\
{}[2^3] & (\varnothing,[2]) = \chi'_2 \\
{}[2,1^4] & (\varnothing,[1^2]) = \epsilon.
\end{array}
\]}

\section{Proof of Theorem~\ref{thm:main}}
\label{sect:proof}

The proof of Theorem~\ref{thm:main} is quite straightforward: we simply
attempt to carry out the Lusztig--Shoji algorithm.  If the attempt
succeeds, and if the solution satisfies the conditions given in
Section~\ref{sect:intro}, then we will have produced a Springer
correspondence.  In fact, we will find that if the unipotent classes are
not as described in the theorem, then the attempt to calculate
$P_{\chi,\cC}$ fails (there is no solution that is a matrix with entries
which are polynomials with nonnegative integer coefficients).

We define
\[
f_k = \max \{ i \mid \text{$1 \le i \le m/2$ and $\supp i \ge \cC_k$} \}.
\]
In the case that $m$ is even, $r' \in \cS$, and $r \notin \cS$, the theorem states that the inequality $f_k - f_{k+1} \le d_{k+1} - d_k$ must hold.  We actually prove below that, irrespective of whether this inequality holds or not, there is a system of Green functions in which the entries of $P$ are Laurent polynomials with nonnegative integer coefficients, and the entries of $\Lambda$ are polynomials with integer coefficients.

The formula for $P$ given in the previous section holds in this more general setting.  The restriction $f_k - f_{k+1} \le d_{k+1} - d_k$ is an immediate consequence of requiring $P_{\chi,\chi'}$ to be a polynomial divisible by $q^{a_{\supp \chi}}$.

\subsection{Proof outline}

We will study the unipotent classes in increasing order, starting with the trivial class.  For each unipotent class $\cC$, we carry out the following four steps:

\subsubsection{Compute $Y^{\cC}_{\chi,\chi'}$, using known formulas for lower classes}
We assume that we are in the case when there is exactly one highest class, say 
$\cD$, below $\cC$. 
From the definition of $Y^\cC_{\chi,\chi'}$, we have
\begin{align}
Y^{\cC}_{\chi,\chi'} &= \gamma^{-1}\left(\omega_{\chi,\chi'} - \sum_{\cC' < \cC}
P_{\chi, \cC'} \Lambda_{\cC'} P_{\chi',\cC'}^t\right) \notag \\
&= Y^{\cD}_{\chi,\chi'} - \gamma^{-1} P_{\chi,\cD} \Lambda_{\cD} P_{\chi',\cD}^t \notag \\
\intertext{In most cases, the matrices $P_{\chi,\cD}$ and $P_{\chi',\cD}$ each have a single nonzero entry (see~\ref{subsubsect:PiC}), so the above equation reduces to}
Y^{\cC}_{\chi,\chi'} &= Y^{\cD}_{\chi,\chi'} - q^{-2a_{\cD}} P_{\chi,S(\chi)} Y^{\cD}_{S(\chi), S(\chi')}  P_{\chi',S(\chi')} \label{eqn:Ygen}
\end{align}
for suitable characters $S(\chi), S(\chi') \in \cD$.  Once we obtain this formula, the formula for $\Lambda_\cC$ follows immediately.

\subsubsection{Set up equations for finding $P_{\chi,\cC}$}
\label{subsubsect:PiC}
From~\eqref{eqn:P-alg} and~\eqref{eqn:Lambda-Y}, we have
\[
P_{\chi,\cC} Y^{\cC}_{\cC,\cC} = q^{a_\cC} Y^{\cC}_{\chi,\cC}
\]
Selecting one entry from both sides of this equation, we have
\begin{equation}\label{eqn:Pgen}
\sum_{\chi'' \in \cC} P_{\chi,\chi''} Y^{\cC}_{\chi'',\chi'}
= q^{a_\cC} Y^{\cC}_{\chi,\chi'}
\end{equation}
The details of the argument from here on vary, but the main ideas are as follows:
\begin{itemize}
\item Since $P_{\chi,\chi''}$ is a polynomial with nonnegative coefficients, $P_{\chi,\chi''}|_{q=1}$ is a nonnegative integer.  Moreover, this integer is $0$ if and only if $P_{\chi,\chi''}$ is the zero polynomial.
\item Evaluating the entire equation~\eqref{eqn:Pgen} at $q = 1$, with judicious choices of $\chi'$, may allow one to deduce strong constraints on the $\chi''$ with $P_{\chi,\chi''} \ne 0$.
\item If $P_{\chi,\chi''}|_{q=1} = 1$, then $P_{\chi,\chi''}$ is a power of $q$.
\end{itemize}
Typically, one shows that there is a unique $\chi'' \in \cC$ such that $P_{\chi,\chi''} \ne 0$.  Let $S(\chi)$ denote this $\chi''$.  Next, one shows that $P_{\chi,S(\chi)}$ is a power of $q$, say $q^b$.  By reconsidering~\eqref{eqn:Pgen}, one may find an equation relating $b$ to $S(\chi)$.

\subsubsection{Determine the members of $\cC$}
To prove that $\chi \in \cC$, we assume that in fact $\supp \chi > \cC$, and we try to use the work of the previous step to compute $P_{\chi,\cC}$.  (We will have already explicitly determined the members of all lower classes, so we know that $\supp \chi \not< \cC$.)  This calculation leads to a contradiction, from which we deduce that $\chi \in \cC$.

\subsubsection{Give a formula for $P_{\chi,\cC}$}
Knowing the members of $\cC$ often enables us to obtain further conditions on those $\chi$ that actually have $\supp \chi > \cC$.  Using these conditions, one can determine the values of $b$ and $S(\chi)$ (as defined in~\ref{subsubsect:PiC}), and hence obtain a formula for $P_{\chi,\cC}$.

\subsection{The trivial unipotent class}

For the smallest unipotent class, we cannot use~\eqref{eqn:Ygen} to compute $Y^{\cC_{\epsilon}}$.  Rather, we must use~\eqref{eqn:Lambda-alg} and~\eqref{eqn:Ydefn} directly:
\[
\Lambda_{\epsilon,\epsilon} = q^{-2m}(\omega_{\epsilon,\epsilon}) = 1
\qquad\text{and}\qquad
Y^{\cC_\epsilon}_{\chi,\chi'} = \gamma^{-1}\omega_{\chi,\chi'}.
\]
Finally, from~\eqref{eqn:P-alg},
\[
P_{\chi,\epsilon} = q^{-m}(\omega_{\chi,\epsilon})\cdot 1
= q^{-m}q^{m} R(\chi \otimes \epsilon \otimes \epsilon) = R(\chi).
\]

\subsection{The classes with Springer representation $r$ or $r'$}

We will carry out the appropriate calculations for $\cC_{\chi_r'}$.  (Recall that we have assumed that if exactly one of $\chi_r'$ and $\chi_r$ is in $\cS$, then in fact $\chi_r' \in \cS$.)  If $\chi_r \in \cS$ as well, the calculations for $\cC_{\chi_r}$ are identical.

From a similar argument as in~\eqref{eqn:Ygen}, we have
\[
Y^{\cC_{\chi_r'}}_{\chi,\chi'} = \gamma^{-1}\omega_{\chi,\chi'} -
q^{-2m} R(\chi) \gamma^{-1} \omega_{\epsilon,\epsilon}R(\chi')
= (\omega_{\chi,\chi'} - R(\chi)R(\chi'))/\gamma. 
\]
In the following table, we carry out this calculation in all cases.  Throughout, we assume that $1 \le i, j < m/2$.
{\tiny
\[
\begin{aligned}
Y^{\cC_{\chi_r'}}_{00} &= (q^{2m} - q^0q^0)/\gamma = q^m+1 \\
Y^{\cC_{\chi_r'}}_{0j} &= (q^{m+j}+q^{2m-j} - q^0(q^j + q^{m-j}))/\gamma
\\
                    &= q^{m-j} + q^j \\
Y^{\cC_{\chi_r'}}_{0r} &- (q^{3r} - q^0q^r)/\gamma = q^r \\
Y^{\cC_{\chi_r'}}_{ir} &= (q^{3r-i} + q^{3r+i} - (q^i+q^{m-i})q^r)/\gamma
\\
                    &= q^{r+i} \\
\end{aligned}
\qquad
\begin{aligned}    
Y^{\cC_{\chi_r'}}_{ij} &= (q^{m+|i-j|} + q^{m+i+j}+q^{2m-i-j}+q^{2m-|i-j|}
\\
                    &\quad - (q^i+q^{m-i})(q^j+q^{m-j}))/\gamma \\
                    &= (q^{2m-|i-j|}+q^{m+|i-j|}-q^{m-i+j}-q^{m+i-j}\\ 
                    &\quad + q^{m+i+j} - q^{i+j})/\gamma 
                    = q^{m-|i-j|} + q^{i+j} \\ 
Y^{\cC_{\chi_r'}}_{rr} &= (q^{2m} - q^rq^r)/\gamma = q^{2r} \\
Y^{\cC_{\chi_r'}}_{rr'}&= (q^m - q^rq^r)/\gamma = 0
\end{aligned}
\]}
By a fortunate coincidence, we can combine many of these formulas.  Below, we permit $0 \le i, j \le m/2$:
\[
Y^{\cC_{\chi_r'}}_{ij} =
\begin{cases}
q^{m-|i-j|} + q^{i+j} & \text{if $i, j < m/2$,} \\
q^{i+j}               & \text{if $i = r$ or $j = r$.}
\end{cases}
\]
Next,~\eqref{eqn:Pgen} gives
\[
P_{ir'}\cdot q^{2r} = q^{r}Y^{\cC_{\chi_r'}}_{ir'} \quad\text{if $i < r$},
\]
from which it follows that $P_{ir'} = q^i$.

\subsection{General arguments for the class $\cC_N$}

\subsubsection{Formula for $Y^{(N)}$}
From~\eqref{eqn:Ygen}, we have
\begin{align*}
Y^{(N)}_{ij} &= Y^{\cC_{\chi_r'}}_{ij} -
(1-\iota)q^{-2r}P_{ir'}Y^{\cC_{\chi_r'}}_{r'r'}P_{jr'} \\
&=
\begin{cases}
q^{m-|i-j|} + q^{i+j} - (1-\iota)q^{-2r}q^i(q^{2r})q^j & \text{if $i, j <
r$,} \\
q^{i+j} - (1-\iota)\cdot 0 & \text{if $i = r$ or $j = r$}
\end{cases}
\end{align*}

We can now set up the various versions of~\eqref{eqn:Pgen} that we
will require.  Let
\[
\delta =
\begin{cases}
0 & \text{if $r \notin \cC_N$,} \\
1 & \text{if $r \in \cC_N$,}
\end{cases}
\qquad
\delta' =
\begin{cases}
0 & \text{if $r' \notin \cC_N$,} \\
1 & \text{if $r' \in \cC_N$.}
\end{cases}
\]
First, suppose $i < m/2$.  
For all $j \in \cC_N$, $j < r$, we have
\begin{equation}\label{eqn:PNdd}
\sum_{\substack{l \in \cC_N \\ l < r}} P_{il}  (q^{m-|l-j|} + \iota q^{l+j}) + \delta P_{ir} q^{j+r} + \delta' P_{ir'} q^{j+r} = q^{d_N}(q^{m-|i-j|} + \iota q^{i+j}).
\end{equation}
If, in fact, $\delta = 1$, so that $r \in \cC_N$, then we also have
\begin{equation}\label{eqn:PNdf}
\sum_{\substack{l \in \cC_N \\ l < r}} P_{il}  q^{l+r} + P_{ir} q^{2r} + \delta' P_{ir'}\cdot 0 = q^{d_N}q^{i+r}.
\end{equation}
An analogue of this holds if $\delta' = 1$.

On the other hand, if $\supp r > \cC_N$, then we have
\begin{equation}\label{eqn:PNfd}
\sum_{\substack{l \in \cC_N \\ l < r}} P_{rl}  (q^{m-|l-j|} + \iota q^{l+j}) +\delta' P_{rr'} q^{j+r} = q^{d_N}q^{j+r},
\end{equation}
and, additionally, if $\delta' = 1$,
\begin{equation}\label{eqn:PNff}
\sum_{\substack{l \in \cC_N \\ l < r}} P_{rl}  q^{l+r} +  P_{rr'}q^{2r} = q^{d_N}\cdot 0.
\end{equation}

\subsubsection{Calculation of $P_{i,\cC_N}$}
\label{subsubsect:PNgen}

Let us assume now that we have shown that the left-hand side
of~\eqref{eqn:PNdd} reduces to a single nonzero term, indexed by some $l
\in \cC_N$, and that, moreover, the coefficient $P_{il}$ is a power of $q$,
say $P_{il} = q^b$.  (This will require slightly different arguments
depending on $\iota$.)
So~\eqref{eqn:PNdd} reduces to
\begin{align}
q^b(q^{m-|l-j|} + \iota q^{l+j}) &= q^{d_N}(q^{m-|i-j|} + \iota q^{i+j})
\notag \\
q^{b+m-|l-j|} + \iota q^{b+l+j} &= q^{d_N+m-|i-j|} + \iota q^{d_N+i+j}
\label{eqn:PNddS}
\end{align}
for all $j \in \cC_N$, $j < m/2$.  In particular, for $j = d_N$, the preceding equation becomes
\begin{equation}\label{eqn:PNddS2}
q^{b + m - l + d_N} + \iota q^{b+l+d_N} = q^{m - |i-d_N| + d_N} + \iota q^{i + 2d_N}
\end{equation}
If $\iota = 1$, we must decide which term on the left corresponds to each term on the right.  We do this by comparing exponents.  Since $l < m/2$, we evidently have $b+m -l + d_N >b+ l + d_N$.  Similarly,
\[
m - |i-d_N| + d_N =
\left\{
\begin{array}{@{}ll@{\qquad}ll}
m + i        &\text{(if $i < d_N$)}&> i + 2d_N &\text{since $d_N < m/2$,} \\
m - i + 2d_N &\text{(if $i > d_N$)}&> i + 2d_N &\text{since $i < m/2$.}
\end{array}
\right.
\]
Therefore,~\eqref{eqn:PNddS2} implies that
\[
q^{b+m-l+d_N} = q^{m-|i-d_N|+d_N}
\qquad\text{and, if $\iota = 1$,}\qquad
q^{b+l+d_N} = q^{i+2d_N},
\]
and hence
\begin{equation}\label{eqn:PNb}
b = l - |i-d_N|
\qquad\text{and, if $\iota = 1$,}\qquad
b = i + d_N - l.
\end{equation}

\subsection{The class $\cC_N$ when $\iota = 1$}

We will begin by showing that if $m$ is even, then $r, r' \in \cC_N$.  Suppose, for instance, that $\supp r > \cC_N$ instead.  Then,~\eqref{eqn:PNfd} evaluated at $q = 1$ can hold only if $P_{rl} = 0$ for all $l < r$, and $\delta' P_{rr'}|_{q=1} = 1$.  In particular, it is necessarily the case that $\delta' = 1$.  But now these values for $P_{rl}$ and $P_{rr'}$ violate~\eqref{eqn:PNff}.  So it must be that $r \in \cC_N$.  The same argument, with the roles of $r$ and $r'$ reversed, shows that $r' \in \cC_N$.

Now, suppose that $i < m/2$.  If $m$ is even, then
comparing~\eqref{eqn:PNdf} with the analogous equation in which $r$ and
$r'$ are exchanged (recall that $\delta = \delta' = 1$) shows that $P_{ir}
= P_{ir'}$.  Suppose they are both nonzero.  Then,
evaluating~\eqref{eqn:PNdd} at $q = 1$ shows that in fact $P_{ir}|_{q=1} =
P_{ir'}|_{q=1} = 1$ (since we are in the case $\iota = 1$), and,
furthermore, that $P_{il} = 0$ for all $l < r$. $P_{ir}$ and $P_{ir'}$ must
be powers of $q$; say $P_{ir} = P_{ir'} = q^b$. Now~\eqref{eqn:PNdd}
reduces to
\[
2q^bq^{j+r} = q^{d_N+m-|i-j|} + q^{d_N+i+j}.
\]
In order for there to be a $b$ satisfying this equation for all $j \in \cC_N$, $j < r$, it must be that $d_N+m-|i-j| = d_N+i+j$ for all such $j$.  Therefore,
\[
m = i+j+|i-j| = 2\max\{i,j\},
\]
in contradiction to the fact that $i,j < m/2$.  We conclude that $P_{ir} = P_{ir'} = 0$.

We now drop the assumption that $m$ is even, and we return to~\eqref{eqn:PNdd} with the knowledge that the last two terms of the left-hand side vanish in all cases.  Again considering that equation at $q = 1$, we see that there is a unique $l \in \cC_N$, $l < m/2$, such that $P_{il}$ is nonzero.  For that $l$, we know that $P_{il}|_{q=1} = 1$, so $P_{il}$ must be a power of $q$, say $q^b$.  We are therefore in the situation of Section~\ref{subsubsect:PNgen}, and~\eqref{eqn:PNb} holds.  

\subsubsection{Members of $\cC_N$}

We have already shown at the beginning of the section that $r, r' \in \cC_N$.  Now, suppose that $d_N < i < m/2$.  If $i \notin \cC_N$, then we can calculate $P_{il}$ as above, and the formulas in~\eqref{eqn:PNb} say
\[
b = l - i + d_N = i + d_N - l.
\]
But this implies that $i = l$, which is absurd.  Therefore, if $d_N < i < m/2$, it must be that $i \in \cC_N$.  We conclude that
\[
\cC_N =
\begin{cases}
\{d_N, d_N+1, \ldots, (m-1)/2 \} & \text{if $m$ is odd,} \\
\{d_N, d_N+1, \ldots, r-1, r, r'\} & \text{if $m$ is even.}
\end{cases}
\]
In particular, $f_N = \lfloor\frac{m-1}{2}\rfloor$ in all cases.

\subsubsection{Calculation of $P_{i,\cC_N}$}
It remains to consider the case $i < d_N$.  We must determine the unique
$l$ such that $P_{il} \ne 0$, and then we must give a formula for that
$P_{il}$.  The formulas~\eqref{eqn:PNb} now say
\[
b = l - d_N + i = i + d_N - l.
\]
These imply that $l = d_N$, and hence that $b = i$.  Thus, if $i < d_N$ and $j \in \cC_N$, we have
\[
P_{ij} =
\begin{cases}
q^i & \text{if $j = d_N$,} \\
0   & \text{otherwise.}
\end{cases}
\]

\subsection{The class $\cC_N$ when $\iota = 0$}

In this case, we obviously have $\delta' = 0$ (as $r'$ is the Springer representation of a smaller class).  Moreover, since $q^{j+r} = q^{m-|r-j|}$, the equation~\eqref{eqn:PNdd} can be rewritten in the more concise form
\[
\sum_{l \in \cC_N} P_{il} q^{m-|l-j|} = q^{d_N}q^{m - |i-j|}.
\]
Evaluating this at $q = 1$ shows that there is a unique $l \in \cC_N$ such that $P_{il} \ne 0$, and that for that $l$, $P_{il}$ is a power of $q$.  As before, the arguments of Section~\ref{subsubsect:PNgen} hold.  There is now only one term on each side of~\eqref{eqn:PNddS}, and we conclude that $b + m - |l- j| = d_N + m -|i-j|$, or
\begin{equation}\label{eqn:PNbi0}
b = d_N - |i - j| + |l-j| \qquad\text{for all $j \in \cC_N$}
\end{equation}

\subsubsection{Members of $\cC_N$}

Suppose $d_N < i < f_N$.  If $\supp i > \cC_N$, then we could use the above formula to compute $P_{i,\cC_N}$.  Putting $j = d_N$ and $j = f_N$ respectively in~\eqref{eqn:PNbi0}, we find that
\begin{align*}
b &= d_N - (i - d_N) + (l - d_N) = d_N - i + l \\
b &= d_N - (f_N - i) + (f_N - l) = d_N + i - l
\end{align*}
This equations together imply that $i = l$, which is absurd.  Therefore,
\[
\cC_N = \{d_N, d_N + 1, \ldots, f_N\}.
\]

\subsubsection{Calculation of $P_{i,\cC_N}$}

We now know that either $i < d_N$ or $i > f_N$.  Suppose $i < d_N$.  As
above, putting $j = f_N$ into~\eqref{eqn:PNbi0} shows that $b = d_N + i -
l$.  On the other hand, the first formula in~\eqref{eqn:PNb} simplifies to
$b = l - d_N + i$.  These two formulas together imply that $l = d_N$ and $b
= i$. 

Now, suppose that $i > f_N$.  This time, putting $j = f_N$
into~\eqref{eqn:PNbi0} gives 
\[
b = d_N - (i - f_N) + (f_N - l) = d_N -i -l + 2f_N.
\]
On the other hand,~\eqref{eqn:PNb} now says that $b = l - i + d_N$.  These
two together imply that $l = f_N$, and that $b = d_N + f_N - i$.  We
conclude that if $i \notin \cC_N$ and $j \in \cC_N$, then
\[
P_{ij} =
\begin{cases}
q^i & \text{if $i < d_N$ and $j = d_N$,} \\
q^{d_N + f_N - i} & \text{if $i > f_N$ and $j = f_N$,} \\
0 & \text{otherwise.}
\end{cases}
\]

\subsection{The class $\cC_N$ when $\iota = -1$}

The calculations required in this case are nearly identical to those to be
done in Section~\ref{subsect:Ck}.  Indeed, by introducing the notation
$d_{N+1} = f_{N+1} = m$, we see that~\eqref{eqn:PNdd} is the same as the
first equation in~\eqref{eqn:Pk} below, with $k = N$.  Some of the cases
considered below are inapplicable here, since the inequalities $i >
f_{N+1}$ and $j > f_{N+1}$ cannot occur.  To determine $\cC_N$ and
calculate $P_{i,\cC_N}$, however, we simply quote the appropriate portions
of the results from the following section.  We have 
\[
\cC_N = \{d_N, d_N + 1, \ldots, r-1\}
\]
and, if $i < d_N$ and $j \in \cC_N$,
\[
P_{ij} =
\begin{cases}
q^i & \text{if $j = d_N$,} \\
0 & \text{otherwise.}
\end{cases}
\]

\subsection{Larger unipotent classes}
\label{subsect:Ck}

\subsubsection{Formula for $Y^{(k)}$}
From~\eqref{eqn:Ygen}, we have
\[
Y^{(k)}_{\chi,\chi'} = Y^{(k+1)}_{\chi,\chi'} - q^{-2d_{k+1}}
P_{\chi,S(\chi)} Y^{(k+1)}_{S(\chi), S(\chi')} P_{\chi', S(\chi')}.
\]
We will first treat the case $k = N-1$.  If $i, j < d_N$, we have
\begin{align*}
Y^{(N-1)}_{ij} &= q^{m-|i-j|} + \iota q^{i+j} - q^{-2d_N} q^i (q^m + \iota
q^{2d_N}) q^j\\
&= q^{m - |i-j|} - q^{m + i + j - 2d_N}.
\end{align*}
Next, if $i < d_N$ and $f_N < j$, then $\iota = 0$, and we have
\[
Y^{(N-1)}_{ij} = q^{m-(j-i)} - q^{-2d_N} q^i
(q^{m-(f_N-d_N)}) q^{d_N+f_N-j} = 0.
\]
Finally, if $i, j > f_N$, then again $\iota = 0$, and
\begin{align*}
Y^{(N-1)}_{ij} &= q^{m-|i-j|} - q^{-2d_N} q^{d_N+f_N-i}(q^m)q^{d_N+f_N-j}
\\
&= q^{m-|i-j|} - q^{m-i-j+2f_N}.
\end{align*}
Next, we carry out analogous calculations for $Y^{(k)}$,
assuming that the formulas for $Y^{(k+1)}$ and the various $P_{ij}$ for $j
\in \cC_{k+1}$ are known.  If $i, j < d_{k+1}$, then
\begin{align*}
Y^{(k)}_{ij} &= q^{m-|i-j|} - q^{m+i+j-2d_{k+2}} - q^{-2d_{k+1}}q^i (q^m -
q^{m + 2d_{k+1} - 2d_{k+2}})q ^j \\
&= q^{m-|i-j|} - q^{m + i + j - 2d_{k+1}}.
\end{align*}
Next, if $i < d_{k+1}$ and $j > f_{k+1}$, then
\[
Y^{(k)}_{ij} = 0 - q^{-2d_{k+1}} q^i \cdot 0 = 0.
\]
(It should be noted that the $0$ in the second term may arise
in two different ways: if $f_{k+1} \in \cC_{k+1}$, then
$Y^{(k+1)}_{d_{k+1},f_{k+1}} = 0$, but on the other hand, if $f_{k+1}
\notin \cC_{k+1}$, then $P_{j,\cC_{k+1}} = 0$.)  Finally, if $i,j >
f_{k+1}$, then if $f_{k+1} \notin \cC_{k+1}$ (which implies $f_{k+1} =
f_{k+2}$), we have
\[
Y^{(k)}_{ij} = q^{m-|i-j|} - q^{m-i-j+2f_{k+2}} - q^{-2d_{k+1}} \cdot 0
= q^{m-|i-j|} - q^{m-i-j+2f_{k+1}}.
\]
On the other hand, if $f_{k+1} \in \cC_{k+1}$, then
\begin{align*}
Y^{(k)}_{ij} &= q^{m-|i-j|} - q^{m-i-j+2f_{k+2}} \\
&\quad - q^{-2d_{k+1}}
q^{d_{k+1} + f_{k+1}-i}(q^{m} -
q^{m-2f_{k+1}+2f_{k+2}})q^{d_{k+1}+f_{k+1}-j} \\
&= q^{m-|i-j|} - q^{m -i -j +2f_{k+1}}
\end{align*}

\subsubsection{Calculation of $P_{i,\cC_k}$}

Suppose $\supp i > \cC_k$ and $j \in \cC_k$.  From the description of
classes below $\cC_k$, we know that $i, j \notin \{d_{k+1}, d_{k+1}+1,
\ldots, f_{k+1}\}$.  There are therefore four versions of~\eqref{eqn:Pgen}
to consider, depending on whether $i < d_{k+1}$ or $i > f_{k+1}$, and
whether $j < d_{k+1}$ or $j > f_{k+1}$.

Note that if $j < d_{k+1}$, then all terms on the left-hand side
of~\eqref{eqn:Pgen} with $l > f_{k+1}$ vanish, since $Y^{(k)}_{jl} = 0$
for those terms.  Likewise, if $j > f_{k+1}$, then all terms with $l <
d_{k+1}$ vanish.
\begin{gather}
\sum_{\substack{l \in \cC_k\\ l < d_{k+1}}}
P_{il} (q^{m - |l-j|} - q^{m +l+j-2d_{k+1}})
= q^{d_k}(q^{m - |i-j|} - q^{m+i+j-2d_{k+1}}) \;\text{if $i,j <
d_{k+1}$,}  \notag\\
\begin{gathered}
\sum_{\substack{l \in \cC_k\\ l > f_{k+1}}}
P_{il} (q^{m - |l-j|} - q^{m -l-j+2f_{k+1}})
= 0 \qquad\text{if $i <
d_{k+1}$, $j > f_{k+1}$,} \\
\sum_{\substack{l \in \cC_k\\ l < d_{k+1}}}
P_{il} (q^{m - |l-j|} - q^{m +l+j-2d_{k+1}})
= 0 \qquad\text{if $i >
f_{k+1}$, $j < d_{k+1}$,}
\end{gathered} \label{eqn:Pk} \\
\sum_{\substack{l \in \cC_k\\ l > f_{k+1}}}
P_{il} (q^{m - |l-j|} - q^{m -l-j+2f_{k+1}})
= q^{d_k}(q^{m - |i-j|} - q^{m-i-j+2f_{k+1}}) \;\text{if $i, j >
f_{k+1}$.} \notag
\end{gather}

Now, evaluating these equations at $q = 1$ is useless---both sides
vanish---but if we differentiate with respect to $q$ first, we obtain
useful information.  Starting from the first equation in~\eqref{eqn:Pk},
we obtain
\begin{multline*}
\sum_{\substack{l \in \cC_k\\ l < d_{k+1}}} \frac{dP_{il}}{dq} (q^{m -
|l-j|} - q^{m +l+j-2d_{k+1}})
\\
+ \sum_{\substack{l \in \cC_k\\ l < d_{k+1}}} P_{il}((m-|l-j|)q^{m-|l-j|-1}
- (m
+l+j-2d_{k+1})q^{m+l+j-2d_{k+1}-1}) \\
= (d_k+m-|i-j|)q^{d_k+m-|i-j|-1} \\
- (d_k+m+i+j-2d_{k+1})q^{d_k+m+i+j-2d_{k+1}-1}.
\end{multline*}
Evalutating at $q = 1$, we obtain
\[
\sum_{\substack{l \in \cC_k\\ l < d_{k+1}}} P_{il}|_{q=1} (2d_{k+1}-l - j -
|l-j|)
= 2d_{k+1} - i - j -|i-j|.
\]
Now, $l+j+|l-j| = 2\max\{l,j\}$, and likewise for the right-hand side, so we have
\begin{equation}\label{eqn:Pkdd}
\sum_{\substack{l \in \cC_k\\ l < d_{k+1}}} P_{il}|_{q=1} (d_{k+1}-
\max\{l,j\})
= d_{k+1} - \max\{i,j\} \qquad\text{if $i, j < d_{k+1}$.}
\end{equation}
Analogous calculations starting from the other equations in~\eqref{eqn:Pk}
yield
\begin{gather}
\sum_{\substack{l \in \cC_k\\ l > f_{k+1}}} P_{il}|_{q=1}
(\min\{l,j\} - f_{k+1}) 
= 0 \qquad\text{if $i < d_{k+1}$, $j > f_{k+1}$,} \label{eqn:Pkdf} \\
\sum_{\substack{l \in \cC_k\\ l < d_{k+1}}} P_{il}|_{q=1}
(d_{k+1} - \max\{l,j\}) 
= 0 \qquad\text{if $i > f_{k+1}$, $j < d_{k+1}$,} \label{eqn:Pkfd} \\
\sum_{\substack{l \in \cC_k\\ l > f_{k+1}}} P_{il}|_{q=1}
(\min\{l,j\} - f_{k+1}) 
= \min\{i,j\} - f_{k+1} \qquad\text{if $i, j > f_{k+1}$.}
\label{eqn:Pkff} 
\end{gather}
Note that $\min\{l,j\} - f_{k+1} > 0$ for all $l > f_{k+1}$. 
Thus,~\eqref{eqn:Pkdf} implies that if $i < d_{k+1}$, then $P_{il}|_{q=1} =
0$ for all $l > f_{k+1}$, and hence $P_{il} = 0$ for all $l > f_{k+1}$.
Similarly,~\eqref{eqn:Pkfd} implies that if $i > f_{k+1}$, then
$P_{il} = 0$ for all $l < d_{k+1}$.

Now, let $g = \max\{l \in \cC_k \mid l < d_{k+1} \}$, and let $h = \min\{l
\in \cC_k \mid l > f_{k+1}\}$.  (Of course, $h$ might not exist, if
$\cC_k$ contains no members larger than $f_{k+1}$.)  Putting $j = g$
into~\eqref{eqn:Pkdd} and $j = h$ into~\eqref{eqn:Pkff}, we find that
\begin{align}
\sum_{\substack{l \in \cC_k\\ l < d_{k+1}}} P_{il}|_{q=1} (d_{k+1}- g)
= d_{k+1} - \max\{i,g\} \qquad\text{if $i, j < d_{k+1}$,} \label{eqn:Pkdu}
\\
\sum_{\substack{l \in \cC_k\\ l > f_{k+1}}} P_{il}|_{q=1}
(h - f_{k+1}) 
= \min\{i,h\} - f_{k+1} \qquad\text{if $i, j > f_{k+1}$.} \label{eqn:Pkfu}
\end{align}
Note that if $g < i < d_{k+1}$, then no positive integers $P_{il}|_{q=1}$
satisfy~\eqref{eqn:Pkdu}, since $d_{k+1} - g > d_{k+1} - i > 0$.  Thus, if
$i < d_{k+1}$, then it is necessarily the case that $i < g$ as well. 
Now,~\eqref{eqn:Pkdu} says that
\[
\sum_{\substack{l \in \cC_k\\ l < d_{k+1}}} P_{il}|_{q=1} (d_{k+1}- g)
= d_{k+1} - g,
\]
so as usual, there is a unique $l$ such that $P_{il}$ is nonzero, and for
that $l$, $P_{il}$ is a power of $q$, say $q^b$.  The first equation
in~\eqref{eqn:Pk} now reduces to
\[
q^b (q^{m - |l-j|} - q^{m +l+j-2d_{k+1}})
= q^{d_k}(q^{m - |i-j|} - q^{m+i+j-2d_{k+1}}).
\]
By matching exponents of corresponding terms, we see that $b+m -|l-j| =
d_k +m -|i-j|$, and $b+m+l+j-2d_{k+1} = d_k+m+i+j-2d_{k+1}$.  Therefore,
\begin{equation}\label{eqn:Pkbd}
b = d_k + |l - j| - |i - j|
\qquad\text{and}\qquad
b = d_k + i - l.
\end{equation}
Similarly, if $i > f_{k+1}$, it follows from~\eqref{eqn:Pkfu} that $i >
h$.  One then deduces that there is a unique nonzero $P_{il}$, and that it
is of the form $q^b$, where
\begin{equation}\label{eqn:Pkbf}
b = d_k + |l - j| - |i - j|
\qquad\text{and}\qquad
b = d_k - i + l.
\end{equation}

\subsubsection{Members of $\cC_k$}

We wish to show that if $d_k < i < d_{k+1}$, or if $f_{k+1} < i < f_k$,
then $i$ necessarily belongs to $\cC_k$.  As usual, we show this by trying
to calculate $P_{i,\cC_k}$ and deriving a contradiction.

Suppose that $d_k < i < d_{k+1}$.  Putting $j = d_k$ into~\eqref{eqn:Pkbd}
gives
\[
b = d_k + (l - d_k) - (i - d_k) = d_k + l - i
\qquad\text{and}\qquad
b = d_k + i - l.
\]
Together, these imply that $i = l$, which is absurd.  Similarly, if $f_{k+1} < i < f_k$, then we put $j = f_k$ in~\eqref{eqn:Pkbf} to find that
\[
b = d_k + (f_k - l) - (f_k - i) = d_k - l + i
\qquad\text{and}\qquad
b = d_k + i - l,
\]
again deducing that $i = l$.  Thus, if $d_k < i < d_{k+1}$ or $f_{k+1} < i < f_k$, it cannot be the case that $\supp i > \cC_k$.  We already know that $\supp i \not< \cC_k$, so we conclude that
\[
\cC_k = \{d_k, d_k+1, \ldots, d_{k+1}-1\} \cup \{f_{k+1}+1, f_{k+1}+2, \ldots, f_k\}.
\]

\subsubsection{Calculation of $P_{i,\cC_k}$}

We now know that either $i < d_k$ or $i > f_k$.  In the former case, we put $j = d_k$ into~\eqref{eqn:Pkbd} and find that
\[
b = d_k + (l - d_k) - (d_k - i) = i + l - d_k
\qquad\text{and}\qquad
b = d_k + i - l.
\]
Together, these imply that $b = i$ and $l = d_k$.  Similarly, if $i > f_k$, putting $j = f_k$ in~\eqref{eqn:Pkbf} gives
\[
b = d_k + (f_k - l) - (i - f_k) = d_k - l - i + 2f_k
\qquad\text{and}\qquad
b = d_k - i + l,
\]
which implies that $l = f_k$ and $b = d_k + f_k - i$.  We conclude that
\[
P_{ij} =
\begin{cases}
q^i & \text{if $i < d_k$ and $j = d_k$,} \\
q^{d_k+f_k-i} & \text{if $i > f_k$ and $j = f_k$,} \\
0 & \text{otherwise.}
\end{cases}
\]

\section{A preferred set of Springer representations} \label{section:
preferred Springer}

In this section, we will describe the construction of the set of Springer 
representations that we have obtained in \cite{AA}.

Recall first that any reflection subgroup of $I_2(m)$ is isomorphic to 
a group $I_2(d)$ where $d|m$. We set
\[s_i= \begin{bmatrix} 0 & \zeta^i \\ \zeta^{-i} & 0 \end{bmatrix},
\]
and for each divisor $d$ of $m$ (including $d=m$), we identify $I_2(d)$
with the subgroup of $\GL(V)$ generated by $s_0$ and $s_{m/d}$. In the
case when $m/d$ is even, we denote by $I'_2(d)$ the subgroup of $I_2(m)$ 
generated by $s_1$ and $s_{m/d+1}$. The group $I'_2(d)$ is isomorphic to
$I_2(d)$, but is not conjugate to it. Any reflection subgroup of $I_2(m)$
is conjugate to a $I_2(d)$ or to a $I'_2(d)$. 

We will denote the irreducible representations of $I_2(d)$ by
\[
\chi^{(d)}_0, \chi^{(d)}_1, \ldots, \chi^{(d)}_{\lfloor(d-1)/2\rfloor};
\qquad \epsilon^{(d)};
\qquad\text{and, if $d$ is even,} \qquad
\chi^{(d)}_{d/2}, \chi^{(d)\prime}_{d/2},
\]
with $\chi^{(m)}_i=\chi_i$, 
for $0\le i \le \lfloor(m-1)/2\rfloor$,
$\epsilon^{(m)}=\epsilon$,
and, if $m$ is even, $\chi^{(m)\prime}_{r}=\chi'_r$, where $r=m/2$.
 
We have $j_{I_2(d)}^{I_2(m)}(\chi_i^d)=
j_{I_2'(d)}^{I_2(m)}(\chi_i^d)=\chi_i$,
for $i=0,1$. On the other hand, we can fix a choice between $\chi_r$ and
$\chi'_r$ in order that the following holds
\[j_{I_2(d)}^{I_2(m)}(\epsilon^{(d)})=\begin{cases}
\chi_d & \text{if $d \ne m/2$,} \\
\chi_r' & \text{if $d = r=m/2$}
\end{cases}
\]
and
\[j_{I_2'(d)}^{I_2(m)}(\epsilon^{(d)})=\chi_d.\] 

We will define a subset $\cS_{\pref}$ of $\Irr(I_2(m))$ as follows. If $m=2$,
we put 
$\cS_{\pref}=\Irr(I_2(2))$. If $m>2$ is odd, 
\[\cS_{\pref}=\left\{\chi_0,\chi_1,\epsilon\right\}\cup
\left\{\chi_d\;\,|\,\text{$d$ divides $m$ and is a power of a prime
number}\right\}.\]
If $m>2$ is even, and $r=m/2$,
\[\cS_{\pref}=\left\{\chi_0,\chi_1,\epsilon\right\}\cup
\left\{\chi_d\;\,|\,\text{$d\ne r$ divides $m$ and is a power of a prime
number}\right\}\cup\left\{\chi_r'\right\}.\] 
In \cite[Def.~8.5]{AA}, a notion of \emph{pseudoparabolic
subgroup} of  a finite complex reflection group has been introduced. 
Then Theorem~8.13 of \cite{AA} says that the set $\cS_{\pref}$ is the set
of all the irreducible representations of $I_2(m)$ which are obtained by
truncated induction from special representations of pseudoparabolic subgroups.

We write $m$ as $m=p_1^{n_1}p_2^{n_2}\cdots p_k^{n_k}$ where $p_1$, $p_2$,
$\ldots$, $p_k$ are prime numbers such that $p_1<p_2<\cdots<p_k$. The
sequence of integers~(\ref{eq: sequence}) 
$d_0<d_1<\cdots<d_{N-1}<d_N$
with respect to $\cS=\cS_{\pref}$ satifies
\[d_{n_1+n_2+\cdots+n_{i-1}+l}=
p_1^{n_1}p_2^{n_2}\cdots p_{i-1}^{n_{i-1}}p_i^{l-1},
\]
for $1\le i\le k-1$ and $0\le l\le n_i$,
with $N=n_1+\cdots+n_k$ and
$d_N=p_1^{n_1}p_2^{n_2}\cdots p_k^{n_k-1}$ if $m$ is not a power of
$2$ and with $N=n-1$ and $d_N=2^{n-2}$ if $m=2^n$, $n\ge 2$.

\begin{example}
{\rm Assume that $m=2p$ with $p$ an odd prime number. Here we have $N=2$,
$d_2=p=r$ and
$\cS_{\pref}=\left\{\chi_0,\chi_1,\epsilon\right\}\cup
\left\{\chi_2\right\}\cup\left\{\chi_p'\right\}$. 
From Theorems~\ref{thm:main} and ~\ref{thm:max}, we obtain that the maximal 
Springer correspondence with respect to $\cS_{\pref}$ has the form
$\cC_0=\{\chi_0\}$, $\cC_1=\{\chi_1,\chi_p\}$,
$\cC_2=\{\chi_2,\chi_3,\ldots,\chi_{p-1}\}$, $\cC_{\chi_p}'=\{\chi_p'\}$
and $\cC_\epsilon=\{\epsilon\}$. It contains the actual Springer
correspondence of $G_2=I_2(6)$ in good characteristic as a special case (see section 4.1).
}\end{example}

\end{document}